\documentclass[12pt]{article}
\usepackage{amsmath, amsthm, amsfonts}
\usepackage{amssymb}
\usepackage{amscd}
\usepackage{verbatim}

\begin{document}
\newcommand{\eqdef}{\stackrel{\mathrm{def}}{=}}
\newcommand{\hilb}{{\mathcal{H}_{\mathbf{a}}}}
\newcommand{\dom}{{\mathcal D}}
\newcommand{\be}{\begin{equation}}
\newcommand{\ee}{\end{equation}}
\newcommand{\thkl}{\rule[-.5mm]{.3mm}{3mm}}
\newcommand{\pde}[2]{\dfrac{\partial{#1}}{\partial{#2}}}
\newcommand{\mpde}[3]{\dfrac{\partial^{#3}{#1}}{\partial{{#2}^{#3}}}}
\newcommand{\cw}{\stackrel{D}{\rightharpoonup}}
\newcommand{\id}{\operatorname{id}}
\newcommand{\supp}{\operatorname{supp}}
\newcommand{\wlim}{\mbox{ w-lim }}
\newcommand{\ntuple}{\{1,\dots,N\}}
\newcommand{\mymu}{{x_N^{-p_*}}}
\newcommand{\R}{{\mathbb R}}
\newcommand{\A}{{\mathcal{A}}}
\newcommand{\N}{{\mathbb N}}
\newcommand{\Z}{{\mathbb Z}}
\newcommand{\Q}{{\mathbb Q}}
\newcommand{\dx}[1]{\mathrm{d}{#1}}
\newtheorem{theorem}{Theorem}[section]
\newtheorem{corollary}[theorem]{Corollary}
\newtheorem{lemma}[theorem]{Lemma}
\newtheorem{definition}[theorem]{Definition}
\newtheorem{remark}[theorem]{Remark}
\newtheorem{proposition}[theorem]{Proposition}
\newtheorem{conjecture}[theorem]{Conjecture}
\newtheorem{question}[theorem]{Question}
\newtheorem{example}[theorem]{Example}
\newtheorem{Thm}[theorem]{Theorem}
\newtheorem{Lem}[theorem]{Lemma}
\newtheorem{Pro}[theorem]{Proposition}
\newtheorem{Def}[theorem]{Definition}
\newtheorem{Exa}[theorem]{Example}
\newtheorem{Exs}[theorem]{Examples}
\newtheorem{Rems}[theorem]{Remarks}
\newtheorem{Rem}[theorem]{Remark}
\newtheorem{Cor}[theorem]{Corollary}
\newtheorem{Conj}[theorem]{Conjecture}
\newtheorem{Prob}[theorem]{Problem}
\newtheorem{Ques}[theorem]{Question}
\newcommand{\pf}{\noindent \mbox{{\bf Proof}: }}


\renewcommand{\theequation}{\thesection.\arabic{equation}}
\catcode`@=11
\@addtoreset{equation}{section}
\catcode`@=12


\title{Compactness properties and ground states for the affine Laplacian}
\author{Ian Schindler
\\
{\small  University of Toulouse 1}\\
 {\small F-31400 Toulouse, France }\\
{\small ian.schindler@ut-capitole.fr}\\\and Cyril Tintarev\footnote{This author expresses his gratitude to CEREMATH at University of Toulouse - Capitole for their warm hospitality, and to Torbj\"orn Olsson for his  assistance crucial for this visit.}
\\{\small Uppsala University}\\
{\small SE-751 06 Uppsala, Sweden}\\{\small
tammouz@gmail.com}}
\date{}

\maketitle
\newcommand{\dnorm}[1]{\thkl #1 \thkl\,}

\abstract{The paper studies compactness properties of the affine Sobolev inequality of Gaoyong Zhang et al \cite{Zhang-1,Zhang-p} in the case $p=2$, and existence and regularity of related minimizers, in particular, solutions to the nonlocal Dirichlet problems
\begin{equation*}
-\sum_{i,j=1}^{N}(A^{-1}[u])_{ij}\frac{\partial^2u}{\partial x_i\partial x_j}=f \mbox{ in }\Omega\subset\R^N, 
\end{equation*}
and
\begin{equation*}
-\sum_{i,j=1}^{N}(A^{-1}[u])_{ij}\frac{\partial^2u}{\partial x_i\partial x_j}=u^{q-1}\,,\quad  u>0,\mbox{ in }\Omega\subset\R^N,
\end{equation*}
where  $A_{ij}[u]=\int_\Omega\frac{\partial u}{\partial x_i}\frac{\partial u}{\partial x_j}\dx{x}$ and $q\in(2,\frac{2N}{N-2})$.
}
\section{Introduction}
Affine Sobolev inequality of Gaoyong Zhang \cite{Zhang-1,Zhang-p}
\begin{equation}\label{eq:afS}
 J_p(u)\eqdef \left(\int_{S_1} \dfrac{\dx{S_\omega}}{\|\omega\cdot \nabla u\|_p^{N}
} \right)^{-1/N}\ge C\|u\|_{p^*},
\end{equation} 
where $1\le p<N$, $p^*=\frac{pN}{N-p}$, and $\|\cdot\|_p$ denotes the $L^p(\R^N)$-norm,
is a refinement of the limiting Sobolev inequality $\|\nabla u\|_p\ge C\|u\|_{p^*}$ in the sense that $J_p$ is bounded by the gradient norm $\|\nabla \cdot\|_p$ (inequality (7.1) in \cite{Zhang-p} that easily follows from the definition), but not vice versa. Similarities between functionals $J_p$ and $u\mapsto \|\nabla u\|_p$, in addition to dominating the norm of $L^{p^*}(\R^N)$, include the following immediate properties: both functionals are invariant with respect to actions of translations, dilations, and orthogonal rotations, and, furthermore, they coincide on radially symmetric functions.  In addition to that, however, the affine Sobolev functional is invariant with respect to the action of the group $SL(N)$ of unimodular matrices, i.e. $J_p(u\circ T)=J_p(u)$ whenever $\det T=1$. On the other hand, $\sup_{\det T=1}\|u\circ T\|_{\dot H^{1,p}}=\infty$ for any $u\in C_0^\infty(\R^N)\setminus\{0\}$, as it can be easily tested on diagonal matrices, which implies that the inequality $\|\nabla u\|_p\le CJ_p(u)$ is false. 
Applications of the affine Sobolev inequality to information theory are discussed in \cite{Zhang-p}.

In the present paper we study the case $p=2$, where there is a simple relation \eqref{eq:asinf} between the affine Sobolev functional $J_2$ and the gradient norm (this connection is cursively mentioned on p. 20 of \cite{Zhang-p}).This relation yields a one-line proof (see \eqref{eq:proofSob} below) of the affine Sobolev inequality \eqref{eq:afS} for this case. 

The main objective of this paper is to study compactness properties of the affine Sobolev inequality and existence of minimizers in variational problems involving the functional $J_2$. We prove that, similarly to Sobolev embeddings,
the set $\{u\in H_0^1(\Omega), J_2(u)\le 1\}$ is compact in $L^p(\Omega)$,
$1\le p< \frac{2N}{N-2}$, whenever $\Omega$ is a bounded domain (or, more generally, unbounded domains of the null-flask type defined below). The method of the proof is, however, different from the classical Sobolev case and is based on the concentration compactness argument, more specifically, on  profile decompositions of functions with the bounded $J_2$. This profile decomposition is then used to study existence of solutions of variational problems involving $J_2$. 
In Section 2 we outline some basic properties of the functional $J_2$. In Section 3 we study compactness properties of $J_2$ and some simple variational problems. A short Section 4 presents profile decompositions for sequences with a $J_2$-bound. Section 5 return to variational problems, which are handled with help of the profile decomposition. In Section 6 we list some open problems. Appendix contains another proof of the affine Sobolev inequality for the case $p=2$ and, for convenience of the reader, cites profile decomposition theorems for $H^{1,2}(\R^N)$ and $\dot H^{1,2}(\R^N)$.

\section{Elementary properties of the affine Sobolev functional}
\subsection{An equivalent definition of $J_2$ }
Invariance of the functional $J_p$ with respect to actions of unimodular matrices is immediate from the following identity from \cite{Zhang-p} (easily derived by radial integration): 
\begin{equation}\label{eq:inv}
 J_p(u)= \left(\frac{1}{(N-1)!} \int_{\R^N} e^{-\|\xi \cdot \nabla u\|_p} \mathrm{d}\xi \right)^{-1/N}.
\end{equation}
In what follows we always assume $p=2$ and $N>2$.  If we set
\begin{equation}\label{eq:Ascript}
\A_{i,j}[u](x) \eqdef \pde{u}{x_i} \pde{u}{x_j},
\end{equation}
we can represent the $L^2$-norm in \eqref{eq:afS} as
\begin{equation}
\|\xi\cdot\nabla u\|_2^2=\int_{\R^N}   \A[u](x) \xi\cdot \xi
\,\mathrm{d}x,\; \xi\in \R^{N}.
\end{equation}
Let now
\begin{equation}\label{eq:straightA}
A_{i,j}[u]\eqdef \int_{\R^N}\A_{i,j}[u](x)\dx{x}.
\end{equation}
Substituting \eqref{eq:Ascript} into \eqref{eq:inv} and taking $\eta=A[u]^{1/2}\xi$, we have
\begin{eqnarray*}
\int_{\R^N} e^{-\|\xi \cdot \nabla u\|_2} \mathrm{d}\xi & = & \int_{\R^N}  
  e^{- (\int_{\R^N}   \A[u](x) \xi\cdot \xi
  	 \,\mathrm{d}x)^{1/2}} \mathrm{d}{\xi} \\
= \int_{\R^N}  
  e^{- (A[u]\xi\cdot \xi)^\frac12 }\mathrm{d} \xi &= & \int_{\R^N}e^{- |\eta|} (\det A[u])^{-1/2} \mathrm{d} \eta \\
&=& \omega_N (N-1)!(\det A[u])^{-1/2},
\end{eqnarray*}
where $\omega_N$ is the area of a unit sphere in $\R^N$. 
We conclude that
\begin{equation}\label{eq:ournorm}
J_2(u)=\omega_N^{-1/N}(\det A[u])^{1/2N}.
\end{equation}
Note that this expression presumes that the matrix $A[u]$ is well-defined, which is the case if and only if $\nabla u\in L^2$.
In what follows we will fix the  domain of $J_2$ as  $\dot H^{1,2}(\R^N)$.

We will also consider below a functional 
$$J_{2,\Omega}(u)\eqdef \omega_N^{-1/N} (\det A_{\Omega}[u])^{1/2N}$$ where
$A_{\Omega}[u]=\int_{\Omega}\A_{i,j}[u](x)\dx{x}$, $\Omega\subset\R^N$ is an open set, and $u\in H^{1,2}(\Omega)$. Note that 
if $J_{2,\Omega}(u)=0$, and $\Omega$ is convex, then there is a  family  of parallel hyperplanes, such that $u$ is constant on their intersection with $\Omega$.
\subsection{Reduction to the gradient norm}
We would like to characterize the behavior of the matrix \eqref{eq:straightA} relative to action of unimodular matrices.

\begin{lemma}\label{lem:transf}
	Let $T \in SL(N)$ and let $u \in \dot H^{1,2}(\R^N)$. Then \begin{equation}\label{eq:transf}
	A[u \circ T] = T^* A[u]\, T. 
	\end{equation}
	In particular for every  $u \in  \dot H^{1,2}(\R^N)$ there is a $T_0 \in O(N)$ such that $A[u \circ T_0]$ is diagonal, and a  $T \in SL(N)$ such that 	$A[u \circ T]=\det(A[u])I$ and 
	\begin{equation}
\det A[u]^{1/2N} =\det A[u\circ T]^{1/2N} = \frac{1}{\sqrt{N}}\|\nabla (u \circ T)\|_2.
	\end{equation}
\end{lemma}
\begin{proof}Equation \eqref{eq:transf} follows by elementary computation from the change of variable $Tx=y$, taking into account that $\partial_iu(Tx)\partial_ju(Tx)=[T^*\A[u](y)T]_{ij}$ and $\dx{x}=\dx{y}$. A suitable $T_0 \in O(N)$ makes $T_0^* A[u]\, T_0$ a diagonal matrix. 
	
	Applying the same transformation once again, with a diagonal unimodular matrix  $T'=\det(A[u \circ T_0])^{1/2}A[u \circ T_0]^{-1/2}$, we get  $A[u \circ T_0T']=\det(A[u \circ T_0])I=\det(A[u])I$. The last assertion follows once we note that 
	$\|\nabla u\circ T_0T'\|_2^2=N\det(A[u])^{1/N}$, since the latter expression is the trace of the diagonal matrix $A[u \circ T_0T']$ with $N$ equal eigenvalues. 	
\end{proof}
\begin{corollary} If $u\in \dot H^{1,2}(\R^N)$, then
\begin{equation}\label{eq:asinf}
	J_2(u)=\frac{\omega_N^{-1/N}}{\sqrt{N}}\min_{T\in SL(N)}\|\nabla(u\circ T)\|_2.
\end{equation}
\end{corollary}
\begin{proof}
Since for any $v\in \dot H^{1,2}(\R^N)$, $\|\nabla v\|_2^2=\mathrm{tr}A[v]$  
 the inequality between the arithmetic and geometric mean gives $\det A[u]^\frac{1}{N}\le \frac{1}{{N}}\|\nabla(u\circ T)\|_2^2$ for any $v\in \dot H^{1,2}(\R^N)$ and $T\in SL(N)$. By Lemma~\ref{lem:transf} the minimum is attained.
\end{proof}
In view of \eqref{eq:asinf} it is convenient to change the scalar multiple in the definition of the "energy" functional associated with $J_2$. Namely, we introduce
\begin{equation}
E_2(u)\eqdef N\det A[u]^{1/N}=N\omega_N^{2/N}J_2(u)^2.
\end{equation} 
In particular, with such normalization, \eqref{eq:asinf} becomes
\begin{equation}\label{eq:asinf2}
	E_2(u)=\min_{T\in SL(N)}\|\nabla(u\circ T)\|_2^2,
\end{equation}
and $E_2(u)=\|\nabla u\|_2^2$ on all radial functions in $\dot H^{1,2}(\R^N)$. We also introduce an analogous functional $E_{2,\Omega}$ defined on $H^{1,2}(\Omega)$.

\subsection{Proof of the affine Sobolev inequality}
The affine Sobolev inequality \eqref{eq:afS} for $p=2$ can be now easily derived from  the usual Sobolev inequality and \eqref{eq:asinf}:
\begin{equation}\label{eq:proofSob}
\|u\|_{2^*}=\inf_{T\in SL(N)}\|u\circ T\|_{2^*}\le C\inf_{T\in SL(N)}\|\nabla(u\circ T)\|_2=C J_2(u)
\end{equation}
See Appendix for an alternative proof.
\subsection{The affine Laplacian}
Let $\Omega\subset\R^N$ be a domain.
By analogy with the $p$-Laplacian which equals the Frechet derivative of $-\frac{1}{p}\int|\nabla u|^p$, 
we may also define the affine Laplace operator $\Delta_A(u)$ by differentiation of  $-\frac12 E_2$ in a suitable space, e.g. in $\dot H^{1,2}_{0}(\Omega)$ for the Dirichelt affine Laplacian or in $H^{1,2}(\Omega)$ for the Neumann affine Laplacian. Since  
$A_{i,j}[u]'= \left( \int_{\Omega} \nabla_i u \nabla_j u\, \dx{x} \right)' =-2 (\nabla_i \nabla_j u)_{ij}$, we have, formally,  
$$\det A[u]' = \det A[u] \; \,\mathrm{tr} (A^{-1}[u]A[u]') = -2\det A \; \mathrm{tr} (A^{-1}[u]u''),$$ 
where $u''(x)$ is the Hessian of $u$, i.e.  the matrix with components  $\nabla_i \nabla_ju(x)$. Then
\begin{align} 
\Delta_A(u)& = &-\frac{N}2 (\det A[u]^{\frac1N})'  \nonumber \\
 & = &
-\frac12(\det A[u])^{\frac1N-1} (\det A[u])'\nonumber \\
& = & (\det A[u])^{\frac1N} \mathrm{tr} (A^{-1}[u]u'').
\end{align} 
It is easy to see that for any $u\in \dot H_0^{1,2}(\Omega)$ this expression is a Frechet derivative of $-\frac12 E_2$ and that $E_2\in C^1(\dot H_0^{1,2}(\Omega))$.  In what follows the notation $\Delta_A$ will be reserved for the affine Dirichlet Laplacian, that is, for the Frechet derivative above.

We have the following elementary identity:
 \begin{equation}
   \label{eq:affine_Laplacian0}
    \Delta_A (u \circ S)=  \Delta_A (u)\circ S,\; S\in SL(N).
    \end{equation} 
  If $T\in SL(N)$ is as in the last assertion of Lemma~\ref{lem:transf}, i.e. $A[u\circ T]$ is a multiple of identity, then we have
   \begin{equation}
   \label{eq:affine_Laplacian}
    (\Delta_A (u)) \circ {T}=  \Delta (u\circ T).
    \end{equation} 
Consequently, both the strong and the weak maximum principle apply to classical solutions of $\Delta_A (v) =f$, exactly in the same form  as  for the classical Laplacian. 
On the other hand, we have a different comparison principle. 
\begin{proposition}[Comparison principle]
	Let $u_1, \ u_2 \in \dot H^{1,2} (\R^N)$   be classical solutions to 
	\begin{equation}
	\Delta_A (u_i) = f_i, \; \; i=1,2,
	\end{equation}
	in $\R^N$.  
	 Let $T_i \in SL(N)$ be as in the second assertion of Lemma~\ref{lem:transf} relative to $u_i$. 
	If $f_1\circ T_1^{-1} \ge f_2\circ T_2^{-1}$ in $\R^N$, then $u_1\circ T_1^{-1} \le u_2\circ T_2^{-1}$ in $\R^N$. 
\end{proposition}
 \begin{proof}
By \eqref{eq:affine_Laplacian}, 
$$\Delta(u_1\circ T_1^{-1} - u_2\circ T_2^{-1})=f_1\circ T_1^{-1} - f_2\circ T_2^{-1},$$ 
and apply the maximum principle for the usual Laplacian. 
\end{proof}

\subsection{Friedrichs and Poincar\'e inequalities}
The Friederichs inequality for the affine Sobolev norm follows from the 
following elementary consequence of the first fundamental theorem of calculus. Let $\Omega\subset\R^N$ be a bounded convex domain. For each $u\in C_0^\infty(\Omega)$,
\begin{equation}
\int_\Omega |u|^2 \dx{x}\le 
C \int_\Omega |\nabla_i u|^2 \dx{x},
\end{equation}
$i=1,\dots, N$. Assuming that $u$ has a support in a subset of $\Omega$ we may drop the requirement of convexity.
Due to Lemma~\ref{lem:transf}  we may assume without loss of generality that $A[u]$ is a diagonal matrix. Taking the product over $i$ in the inequality above gives 
$$\|u\|_{2,\Omega}\le CJ_{2,\Omega}(u).$$

An immediate analog of the Poincar\'e inequality,
 $$J_{2,\Omega}(u)^2+\left(\int_\Omega u\right)^2\ge C\|u\|_{2,\Omega}^2,$$ 
 is false, since the left hand side will vanish on any nonzero function dependent only on $x_1$, whose integral over $\Omega$ is zero.
\section{Some variational problems}
In what follows the norm of a matrix $T$ will be denoted as $|T|$. We note that a sequence $(T)_k\subset SL(N)$ is either unbounded in norm, or has a  subsequence convergent to a matrix in $SL(N)$. 
\subsection{Affine Laplace equation}
\begin{definition}
We shall say that a function $f\in L^{\frac{2N}{N+2}}(\Omega)$ is of class $L_A(\Omega)$ if for any sequence $(T_k)\in SL(N)$, $|T_k|\to\infty$, one has\\ $f\circ T_k|_\Omega\to 0$ in $L^{\frac{2N}{N+2}}(\Omega)$.
\end{definition}
In particular, if $\Omega$ is bounded,  $L_A(\Omega)=L^{\frac{2N}{N+2}}(\Omega)$, and if $\Omega=\R^N$, $L_A(\Omega)=\{0\}$. 
\begin{theorem}
\label{thm:Laplace}	Let $\Omega \subset \R^N$ be a domain with a piecewise-$C^1$ boundary.  If $f\in L_A(\Omega)$, then the infimum
\begin{equation}
	\label{eq:Lmin}
	\kappa_f \eqdef	\inf_{u \in \dot H_{0}^{1,2}(\Omega)}\frac12 E_2(u) -\int_{\Omega} f(x)u(x)\dx{x}	
	\end{equation}
is attained.	
If, additionally, $f\in L^2(\Omega)$, then this minimizer is a classical solution of
	\begin{equation}
	\label{eq:LE}
\Delta_A(u)(x)+f(x)=0,\;x \in \Omega.
\end{equation}
	
\end{theorem}
\begin{proof}
Note first that $\kappa_f<0$. Indeed, let $w\in C_0^1(\Omega)$ be such that $\int_\Omega fw\dx{x}<0$. Then for $t>0$ sufficiently small the functional in \eqref{eq:Lmin} will have negative values, since the first term is quadratic in $t$.

By \eqref{eq:asinf2} we can rewrite \eqref{eq:Lmin} as
\begin{equation}
	\label{eq:Lmin1}
	\kappa_f =	\inf_{v \in \dot H_{0}^{1,2}(T\Omega), \, T\in SL(N)} \frac{1}{2}\int_{T\Omega}|\nabla v|^2\dx{x} -\int_{\Omega} f(x)\,v(Tx)dx{x}.	
	\end{equation}

Let $((v_k, T_k))_{k\in\N} \subset C_{0}^{\infty}(T_k\Omega)\times SL(N)$ be a minimizing sequence for \eqref{eq:Lmin1}. Consider $(v_k)$ as a sequence in $\dot H^{1,2}(\R^N)$. 
Assume first that $|T_k|\to\infty$. Then, since $f\in L_A(\Omega)$, we have $\kappa_f\ge 0$, which is false.  Consequently, we have, on a renamed subsequence, $T_k\to T\in SL(N)$ and $v_k\rightharpoonup v$ in $\dot H^{1,2}(\R^N)$ with $v=0$ outside of $T\Omega$, which means that $u\eqdef v\circ T^{-1}\in \dot H_{0}^{1,2}(\Omega)$ is a required minimizer.  
Equation \eqref{eq:LE} (in the weak sense) follows, and the regularity of the solution is a consequence of the standard elliptic regularity.
\end{proof}
\begin{remark}
In absence of a simple comparison principle we have no immediate uniqueness theorem.
\end{remark}

\subsection{Isoperimetric problems, the case of $\R^N$}

\begin{theorem}
	The minimal values in the problems
	\begin{equation}
	\label{eq:crit}
		\inf_{u \in \dot H^{1,2}(\R^N), \|u\|_{2^*} =1} E_2(u),
	\end{equation}
	and 
		\begin{equation}
		\label{eq:subcrit0}
			\inf_{u \in H^{1,2}(\R^N), \|u\|_{p} =1}E_2(u)+\|u\|_2^2, \,2<p<2^*,
		\end{equation}
	are attained.
\end{theorem}
The first part of the theorem was proved in \cite{Zhang-p} for the affine p-Laplacian, but for the case the case $p=2$ there is an elementary proof that we include here.  
\begin{proof}
	By \eqref{eq:asinf2}, for every $u \in \dot H^{1,2}(\R^N)$ there is $T \in SL(N)$ such that  $E_2(u)= \| \nabla (u \circ T) \|_2^2$. Therefore 
\begin{equation}
\inf_{u \in \dot H^{1,2}(\R^N), \|u\|_{2^*} =1} E_2(u) = \inf_{u \in \dot H^{1,2}(\R^N), \|u\|_{2^*} =1} \| \nabla u \|_2^2
\end{equation}
and 
\begin{equation}
\inf_{u \in H^{1,2}(\R^N), \|u\|_{p} =1}E_2(u)+\|u\|_2^2=
\inf_{u \in H^{1,2}(\R^N), \|u\|_{p} =1}\| \nabla u \|_2^2+\|u\|_2^2.
\end{equation}
Thus the infimum \eqref{eq:crit} is uniquely attained at the Talenti-Bliss minimizer under any combined action of dilations and affine transformations, and the infimum \eqref{eq:subcrit0} is attained at well-known unique radial minimizer under any combined action of translations and affine transformations.
\end{proof}

\subsection{Affine-null domains and compactness in $L^{p}$}
In what follows $|\Omega|$ will denote the Lebesgue measure of a set.
Recall the definition of the lower limit for a sequence $(X_k)$ of sets:
\[
\liminf X_k\eqdef\bigcup_{n\in\N}\bigcap_{k\ge n}X_k.
\]
\begin{definition}
	A subset $\Omega$ of $\R^N$ will be called affine-null set if  for any sequences $(T_k)\subset SL(N)$ and $(y_k)\subset\Z^N$, such that $|T_k|+|y_k|\to\infty$,  
	\begin{equation}
	\label{eq:no_profiles}
		|\liminf T_k^{-1}(\Omega-y_k)|=0.
	\end{equation}

\end{definition}
Note that any bounded set is affine-null.   An example of an unbounded affine null set is $\{ (x_1,\bar x) \in \R \times \R^{N-1} : | \bar x  | < e^{-x_1^2}\}$. Not every null set relative to the group of shifts alone (i.e. $\forall (y_k)\subset\R^N$ $|\liminf(\Omega-y_k)|=0$) is affine-null. In particular, the set $\{ (x_1,\bar x) \in \R \times \R^{N-1} : | \bar x  | < (1+\log {|x_1|})^{-1}\}$ is shifts-null but not affine-null.
 
\begin{theorem} 
	\label{thm:null-compact} 
	Let $\Omega\subset\R^N$ be an affine-null domain $[$for example, a bounded domain$]$.
Then the set $B_1=\{u\in H^{1,2}_{0}(\Omega);\;E_2(u)\le 1\}$ is relatively compact in $L^p(\Omega)$, $2< p< 2^*$.  
\end{theorem}
Note that the set $B_1$ is not bounded in $H^{1,2}_{0}(\Omega)$.
\begin{proof}
	Let $(u_k) \subset B_1$ and consider it as a sequence in $H^{1,2}(\R^N)$.
	Let $T_k\in SL(N)$ be as in \eqref{eq:qq}.
	Let $v_k=u_k \circ T_k$. Then $(v_k)$ is a bounded sequence in $H^{1,2}_{0}(\Omega)$, which we will consider as a sequence in 
	$H^{1,2}(\R^N)$. If $|T_k| \to \infty$ then by \eqref{eq:no_profiles}, $v_k(\cdot-y_k)\rightharpoonup 0$  in $H^{1,2}(\R^N)$ for any sequence $(y_k)\subset\R^N$ (for details see the argument in the proof of Lemma~4.1 in \cite{FT}), which implies (e.g. by Proposition~\ref{acc-shifts}) that $v_k\to 0$ in $L^p$, $2<p<2^*$, and thus $u_k\to 0$ in $L^p$. 
   Otherwise, there is a renamed subsequence of $(T_k)$ convergent to some $T\in SL(N)$. Passing again to a renamed weakly convergent subsequence we may assume that $v_k\rightharpoonup v$ in $H^{1,2}(\R^N)$,
   and thus $u_k\rightharpoonup  v\circ T^{-1}$ in $H_0^{1,2}(\Omega)$. On the other hand, from \eqref{eq:no_profiles} we can infer that for any sequence $(y_k)\subset\R^N$, $(v_k-v)(\cdot-y_k)\rightharpoonup 0$ in $H^{1,2}(\R^N)$ and thus, setting $u\eqdef  v\circ T^{-1}$, $\|u_k-u\|_p\le\|v_k-v\|_p+\|u\circ T -u\circ T_k\|_p\to 0$. 
	\end{proof}
\subsection{A semilinear problem in an affine null domain}
\begin{theorem}\label{thm:bddDom} 
	Let $\Omega \subset \R^N$ be an affine-null  domain $[$for example, a bounded domain$]$ with a piecewise-$C^1$-boundary.  Then the minimum in the problem
	\begin{equation}\label{eq:bddDom} 
\kappa_p =	\inf_{u \in H_{0}^{1,2}(\Omega), \|u\|_{p,\Omega} =1} E_2(u),\;2<p<2^*,
	\end{equation}
	is attained.
\end{theorem}

\begin{proof}
Let $(u_k) \subset H_{0}^{1,2}(\Omega)$ be a minimizing sequence. Consider it as a sequence in $H^{1,2}(\R^N)$.  Let $T_k\in SL(N)$ be as in \eqref{eq:qq}. Repeating the argument in the proof of Theorem~\ref{thm:null-compact}, we may assume, for a suitable renamed subsequence,  that either $|T_k|\to \infty$ and then $u_k\to 0$ in $L^p$, or $T_k\to T\in SL(N)$, and $u_k$ converges weakly in $ H_{0}^{1,2}(\Omega)$ as well as in $L^p(\Omega)$ to some $u$.
The former case is ruled out, since by assumption $\|u_k\|_{p,\Omega}=1$. In the latter case, lower semicontinuity of the norm implies that $\|\nabla u\|_2^2\le \kappa_p$. Then by \eqref{eq:asinf2} $E_2(u)\le \kappa_p$, and thus $u$ is necessarily a minimizer.  	
\end{proof}
\begin{corollary}\label{cor:reg}
	Let $\Omega \subset \R^N$ be a bounded domain with a piecewise $C^1$-boundary.  Then \eqref{eq:bddDom} has a minimizer that, up to a scalar multiple, is a smooth positive classical solution of the boundary problem
\begin{equation} \label{eq:eq}
-\sum_{i,j=1}^{N}(A^{-1}[u])_{ij}\frac{\partial^2u}{\partial x_i\partial x_j}=u^{p-1} \mbox{ in }\Omega,\,u|_{\partial\Omega}=0.
\end{equation}
\end{corollary}
\begin{proof}
Note that if $u\in H_0^1(\Omega)$ is a minimizer for \eqref{eq:bddDom}, then 
so is $|u|$ by \eqref{eq:asinf2}:
\begin{align*}
\kappa_p =	\inf_{u \in H_{0}^{1,2}(\Omega), \|u\|_{p,\Omega} =1} E_2(u)
\\
=\inf_{u \in H_{0}^{1,2}(\Omega), \|u\|_{p,\Omega} =1, T\in SL(N)} \|\nabla(u\circ T)\|_2^2
\\
=\inf_{u \in H_{0}^{1,2}(\Omega), \|u\|_{p,\Omega} =1, T\in SL(N)} \|\nabla|u\circ T|\|_2^2
\\
=\inf_{u \in H_{0}^{1,2}(\Omega), \|u\|_{p,\Omega} =1}  E_2(|u|),
\end{align*}
 so we can without loss of generality assume that $u\ge 0$. Then, for some $\lambda>0$, the function $u$ satisfies, in the weak sense, 
\begin{equation} \label{eq:eq2}
-\sum_{i,j=1}^{N}(A^{-1}[u])_{ij}\frac{\partial^2u}{\partial x_i\partial x_j}=\lambda u^{p-1} \mbox{ in }\Omega.
\end{equation}
Note that $A[u]^{-1}$ is a positive constant matrix, as an inverse of a positive matrix, so the standard elliptic regularity and the bootstrap argument yield the smoothness of the solution. The solution is strictly positive by maximum principle for uniformly elliptic operators.
Finally, note that the left hand side of \eqref{eq:eq2} is of homogeneity $-1\neq p-1$, so a suitable scalar multiple of $u$ satisfies  \eqref{eq:eq}.    
\end{proof}

\section{Profile decompositions}


In this section we outline concentration behavior of sequences with bounded values of $E_2$ (note that they are not necessarily bounded in the Sobolev norm). 


\begin{theorem}
	\label{thm:ourPD}
		Let $(u_k)\subset\dot H^{1,2}(\R^N)$ satisfy $E_2(u_k)\le C$. There exist $(T_k)\subset SL(N)$, $w^{(n)}  \in
	\dot H^{1,2}(\R^N)$, $(y_{k} ^{(n)})_{k\in\N} \subset \R^N$, $(j_{k} ^{(n)})_{k\in\N}\subset \Z$ with $n \in\N$, and disjoint sets
	$\N_0,\N_{+\infty},\N_{-\infty}\subset\N$, such that, for a
	renumbered subsequence of $(u_k)$,
	\begin{eqnarray}
	\label{w_n}
	&&
	2^{-\frac{N-2}{2}j_k^{(n)}}u_k(T_k(2^{-j_k^{(n)}}\cdot+y_k^{(n)}))\rightharpoonup 	w^{(n)},\;
	n\in\N,
	\\
	\label{separates}
	&&
	|j_{k}^{(n)} -  j_{k}
	^{(m)}|+|2^{j_k^{(n)}}(y_{k} ^{(n)} - y_{k} ^{(m)})|\to \infty
	\mbox{ for } n \neq m,
	\\
	\label{norms}
	&&
	\sum_{n\in \N} \|\nabla w ^{(n)}\|_2^2 \le
  \liminf E_2(u_k),
	\\
	\label{BBasymptotics}
	&&
	u_{k}  -  \left[\sum_{n\in\N}
	2^{\frac{N-2}{2}j_k^{(n)}} w^{(n)}(2^{j_k^{(n)}}(\cdot-y_k^{(n)}))\right]\circ T_k^{-1}
	\to 0\; \mbox{in } L^{2^*},
	\end{eqnarray}
	 and the series in the square brackets above converges in $\dot H^{1,2}(\R^N)$ unconditionally and uniformly
	with respect to $k$.
	\par Moreover,
	$1\in\N_0$, $y_{k} ^{(1)} =0$; $j_k^{(n)}=0$ whenever $n\in\N_0 $;
	$j_k^{(n)}\to -\infty$ (resp. $j_k^{(n)}\to +\infty$) whenever
	$n\in\N_{-\infty}$ (resp. $n\in\N_{+\infty}$); and $y_k^{(n)}=0$
	whenever $|2^{j_k^{(n)}}y_k^{(n)}|$ is bounded.
\end{theorem}

\begin{proof}
Let $T_k\in SL(N)$ such that, according to Lemma~\ref{lem:transf} 
\begin{equation}
\label{eq:qq}
E_2(u_k) =E_2(u_k\circ T_k) = \|\nabla (u_k \circ T_k)\|_2^2.	
\end{equation}
 Let $v_k=u_k \circ T_k$ and apply Theorem~\ref{thm:PDSob} from Appendix. To conclude the proof of Theorem~\ref{thm:ourPD} 
it remains to note that \eqref{BBasymptotics*} gives \eqref{BBasymptotics} by composing the left and the right hand side with $T^{-1}_k$ on the right, and that the right hand side of \eqref{norms*} yields the right hand side of \eqref{norms} by \eqref{eq:qq}.
\end{proof}

A analogous decomposition for sequences with bounded $E_2+\|\cdot\|_2^2$ can be derived in a completely analogous way from Proposition~\ref{acc-shifts} in Appendix: 

\begin{proposition}
	\label{acc-shiftsA}
	Let $(u_{k})\in H^{1,2}(\R^N)$ be a sequence such that $E_2(u_k)+\|u_k\|_2^2\leq C$. There
	exist $w^{(n)}  \in
	H$, $(T_k)\subset SL(N)$, and $(y_{k} ^{(n)})_{k\in\N} \subset \Z^N$, $y_{k} ^{(1)} =0$, $n\in\N$,
	such that, on a renumbered subsequence,
	\begin{eqnarray}
	\label{w_n-shiftsA} && u_k(T_k(\cdot+y_{k} ^{(n)}))\rightharpoonup w^{(n)},
	\\
	&&\label{separates-shiftsA} |y_{k} ^{(n)}-y_{k} ^{(m)}|  \mbox{ for
	} n \neq m,
	\\
	&&\label{norms-shiftsA} \sum_{n \in \N} \|w ^{(n)}\|_{H^{1,2}}^2 \le
	\limsup \|u_k\|_{H^{1,2}}^2,
	\\
	&&\label{BBasymptotics-shiftAs} u_{k} - \left[\sum_{n\in\N}
	w^{(n)}(\cdot-y_{k} ^{(n)})\right]\circ T_k^{-1}  \to 0\, \text{in }
	L^p(\R^N), p\in (2,2^*), 
	\end{eqnarray}
and the series in the square brackets above converges
	in $H^{1,2}(\R^N)$ unconditionally and uniformly in $k$.
\end{proposition}

\section{Affine-flask sets. Poblems with penalty}	

\begin{definition}
	An open subset $\Omega$ of $\R^N$ will be called affine-flask set if  for any $(T_k)\subset SL(N)$ and $(y_k)\subset\Z^N$, such that $|y_k| + |T_k|\to\infty$,  
	there exist a $y \in \Z^N$ and a $T \in SL(N)$ such that 
	\begin{equation}
	\label{eq:flask}
	\left|\liminf T_k^{-1} (\Omega-y_k) \setminus (T \Omega+ y) \,\right| = 0 .
	\end{equation}
\end{definition}
In other words, $\liminf T_k^{-1} (\Omega-y_k)$ is contained, up to a set of measure zero, in the image of $\Omega$ under some affine transformation.

Obviously an affine-null set as well as $\R^N$ are  affine flask sets. 
The union of unit balls $\bigcup_{n\in\N} B_1(n^4e_0)$, $|e_0|=1$, is an affine flask set. If one connects consecutive balls by circular cylinders of corresponding radius $e^{-n}$ that have $\R e_0$ as their common axis, one gets a connected affine flask set.  On the other hand a cylindrical domain with a smooth boundary is an affine flask set only if it is $\R^N$. Indeed, let $\Omega=\R\times\omega$ and let $T_k$ be a diagonal matrix with diagonal entries $k^{1-N},k,\dots,k$. Then $\liminf T_k\Omega=\R^N$. 
\begin{theorem}
\label{thm:flask}	Let $p\in(2,2^*)$ and let $\Omega \subset \R^N$ be an open affine flask set with a piecewise-$C^1$ boundary $[$for example, $\Omega = \R^N$ $]$.  Then the minimum in the problem
	\begin{equation}
	\label{eq:FP}
	\kappa =	\inf_{u \in H_{0}^{1,2}(\Omega): \|u\|_{p,\Omega} =1} E_2(u) + \|u\|_2^2
	\end{equation}
	is attained.
\end{theorem}
\begin{proof}
Let $(u_k) \subset H_0^{1,2}(\Omega)$ be a minimizing sequence. Consider it as a sequence in $H^{1,2}(\R^N)$.
 Let $(T_k) \subset SL(N)$ and let $w^{(n)}$, $n \in \N$, be as in Theorem~\ref{acc-shiftsA}, so we have  $E_2(u_k \circ T_k) =  \| \nabla (u_k \circ T_k) \|_2^2$. 
 From the iterated Brezis-Lieb Lemma (see e.g. \cite{CwiTi}) we have 
 \begin{equation}
 	1 = \|u_k\|_p^p = \sum_n \|w^{(n)} \|_p^p.
 \end{equation}
Let $t_n = \|w^{(n)}\|_p^p$.  
 
 By   \eqref{norms-shiftsA}  
 \begin{align}
 \kappa & = &	\lim E_2(u_k(T_k \cdot -y +y_k^{(n)})) + \|u_k(T_k \cdot -y +y_k^{(n)})\|_2^2\nonumber \\ 
 	& \ge & \sum_{n \in \N}\|\nabla w ^{(n)}\|^2_2 + \| w ^{(n)}\|^2_2 \nonumber \\ 
 	& \ge & \sum_{n \in \N} E_2(w ^{(n)}) + \|w ^{(n)}\|_2^2. \label{eq:energy}
 \end{align}
 Equation \eqref{eq:flask} implies that with some $T^{(n)} \in SL(N)$ and some $y_n \in \R^N$ one has 
 \[
 u_k(T_k ((T^{(n)})^{-1} \cdot -y_n) +y_k^{(n)}) \rightharpoonup w^{(n)} ((T^{(n)})^{-1} (\cdot -y_n)) \in H_0^{1,2}(\Omega).
 \]
 From \eqref{eq:energy} we have 
 \begin{equation}
 \kappa \ge \sum_{n \in \N} \kappa t_n^{2/p},  
 \end{equation}
 which can hold only if $t_n = 0$ for $n \neq m$ and $t_m =1$ with some $m \in \N$.   Consequently $w^{(m)} ((T^{(m)})^{-1} (\cdot -y_m))$ is a minimizer.
\end{proof}
\begin{remark}
Any minimizer for the problem \eqref{eq:flask} is, up to a scalar multiple, a positive smooth solution of the boundary value problem  \begin{equation} \label{eq:eq3}
-\det A[u]^{1/N}\sum_{i,j=1}^{N}(A^{-1}[u])_{ij}\frac{\partial^2u}{\partial x_i\partial x_j}+u=u^{p-1},\;u|_{\partial\Omega}=0.
\end{equation}
The argument copies that of Corollary~\ref{cor:reg} with one modification:
in the proof of the corollary we omitted the scalar factor $\det A[u]^{1/N}$ in the Frechet derivative of the left hand side. We do not omit it here, and as a consequence the left hand side is now of homogeneity $1<p-1$, which allows to replace $u$ by its scalar multiple while setting the Lagrange multiplier to $1$. 
\end{remark}

\begin{theorem}
\label{thm:penalty} Let $p\in(2,2^*)$ and let $V\in L^\infty(\R^N)$ satisfy   $\lim_{|x|\to\infty}V(x)\to 1$ and $V(x)\le 1$, assuming that the latter inequality is strict on a set of positive measure.
 Then the minimum in the problem
	\begin{equation}
	\label{eq:FPp}
	\kappa' =	\inf_{u \in  H^{1,2}(\R^N), \|u\|_{p,\R^N} =1} E_2(u) + \int_{\R^N} V(x) u(x)^2\dx{x}
	\end{equation}
	is attained.
\end{theorem}
\begin{proof} 
Let $(u_k) \subset C_{0}^{\infty}(\R^N)$ be a minimizing sequence. 
 Let $(T_k) \subset SL(N)$ and let $w^{(n)}$, $n \in \N$, be as in
 Theorem~\ref{acc-shiftsA}, so we have  $E_2(u_k \circ T_k) =  \| \nabla (u_k \circ T_k) \|_2^2$.
 From the iterated Brezis-Lieb Lemma we have 
 \begin{equation}
 	1 = \|u_k\|_p^p = \sum_n \|w^{(n)} \|_p^p.
 \end{equation}
Let $t_n = \|w^{(n)}\|_p^p$.  
 
 Let us represent $E_2(u_k \circ T_k)+\int V(x)u_k(x)^2\dx{x}$
 as $\| \nabla (u_k \circ T_k) \|_2^2+\|u_k\circ T_k\|_2^2+\int (V(x)-1)u_k(x\circ T_k)^2\dx{x}$ and note that the last term is weakly continuous in
 $H^{1,2}(\R^N)$.
 
 Assume first that $|T_k|\to\infty$. Then 
 by \eqref{norms-shiftsA} we have
 \begin{align}
 \kappa' & = &	\lim \|u_k(T_k \cdot -y +y_k^{(n)})\|_{H^{1,2}}^2
 	\nonumber \\ 
 	& \ge & \sum_{n \in \N} \|\nabla w ^{(n)}\|^2_2 + \| w ^{(n)}\|^2_2 \nonumber \\ 
 	& \ge & \sum_{n \in \N}E_2(w ^{(n)}) + \| w ^{(n)}\|^2_2
 	\\
 	& \ge & \sum_{n \in \N} \kappa t_n^{2/p}\ge\kappa_p, \label{eq:energy2}
 \end{align}
where $\kappa_p$ is the constant \eqref{eq:bddDom}. Evaluation of the left hand side of \eqref{eq:FPp} at the minimizer of \eqref{eq:FP} gives, however, that $\kappa'<\kappa_p$, which is a contradiction. Consequently, on a suitable renamed subsequence, we have $T_k\to T\in SL(N)$. In this case $u_k\rightharpoonup w^{(1)}\circ T^{-1}$ and
\eqref{norms-shiftsA} gives
\begin{align}
 \kappa' & = &	\lim \|u_k\circ T_k\|_{H^{1,2}}^2+\int (V(x)-1)(w^{(1)}\circ T^{-1})^2\dx{x} 	\nonumber \\ 
 	&\ge &\kappa' t_1^{2/p}+\sum_{n =2}^\infty \kappa t_n^{2/p},
 	\label{eq:energy3}
 \end{align}
which is false unless $t_n = 0$ for $n >1 $ and $t_1 =1$.   Consequently $w^{(1)}\circ T^{-1}$ is a minimizer.
\end{proof}

\section{Open problems}

\begin{enumerate}
	\item The functional 
	\begin{equation}\label{eq:asinfp}
		E_p(u)\eqdef\inf_{T\in SL(N)}\|\nabla(u\circ T)\|^p_p
	\end{equation}
dominates $\|u\|_{p^*}^p$ and is scaling- and $SL(N)$-invariant.  By \eqref{eq:asinf}, if $p=2$, $E_p$ equals $J_p^2$ up to a scalar multiple. 	
	What is the relation between $ E_p$ and $J_p$ for general $p$?
 \item Is there uniqueness for affine Laplace equation \eqref{eq:LE}? What characterizes a pair $(\Omega, f)$ for which the solution of the affine Laplace equation satisfies the ordinary Laplace equation?
	\item An affine Sobolev inequality of Moser-Trudinger type, when $p=N$, has been proved in \cite{cianchi}. What are compactness properties of this embedding? We conjecture, in line with the result from \cite{AdiTi} that when $p=N=2$ the Moser functional $\int_\Omega e^{4\pi u^2}\dx{x}$, where $\Omega\subset \R^2$ is a bounded domain, is weakly sequentially continuous on any sequence $(u_k)\subset  \{u\in H^{1,2}_{0}(\Omega):\;E_2(u)\le 1\}$ unless there exists $(g_k)\in D_A$ such that $E_2(u_k-g_k\mu)\to 0$, where $\mu(x)=\frac{1}{\sqrt{2\pi}}\min\{1,\log\frac{1}{|x|}\}$.
	\item There are many more variational problems  involving the affine Laplacian that can be handled by means of compactness results in this paper, in particular the Brezis-Nirenberg problem and a variety of minimax problems.
\end{enumerate}

\section*{Appendix}
\textbf{1.} Inequality \eqref{eq:afS} in the case $p=2$ can be also easily derived from 
the  intermediate step in the Nirenberg's proof of the usual Sobolev  inequality (\cite{Nirenberg59}, reproduced in the book \cite{GT}):
\begin{equation*}
\int_{\R^N} |u|^{\frac{N}{N-1}} \mathrm{d}x  
\le  C \left(\prod_i \int_{\R^N} |\nabla_i u|\mathrm{d}x\right)^{\frac{1}{N-1}}.
\end{equation*}
Setting $u=|v|^{\frac{2N-2}{N-2}}$ we have
\begin{eqnarray*}
	\int_{\R^N} |v|^{2^*} \mathrm{d}x  & = &  \int_{\R^N} |u|^{\frac{N}{N-1}} \mathrm{d}x  \\
	&  \le  & C \left(\prod_i \int_{\R^N} |\nabla_i u|\mathrm{d}x\right)^{\frac{1}{N-1}} \\
	& =  & C\left(\prod_i \int_{\R^N} |\nabla_i v| |v|^{\frac{N}{N-2}} \mathrm{d}x\right)^{\frac{1}{N-1}} \\
	&\le &  C\left ( \int_{\R^N} |v|^{2^*} \mathrm{d}x \right )^{\frac{N}{2(N-1)}} \left
	( \prod_i \int_{\R^N} |\nabla_i v|^2 \dx{x}\right )^{\frac{1}{2(N-2)}}.
\end{eqnarray*}
Note now that the latter product is $\det A[v]$ whenever $A$ is a diagonal matrix. Since by Lemma~\ref{lem:transf}, any matrix $A[v]$ can be diagonalized by setting $v=w\circ T$ with a suitable $T\in O(N)$, inequality  \eqref{eq:afS} for $p=2$ is proved.  \hfill $\square$
\vskip1cm\noindent
\textbf{2.} The following theorem from \cite{ST} is a trivial refinement of the main theorem in \cite{Solimini} (Sergio Solimini).
\begin{theorem}
	\label{thm:PDSob}
	Let $(v_{k})\subset \dot H^{1,2}(\R^N)$, $N>2$, be a
		bounded sequence. There exist $w^{(n)}  \in
		\dot H^{1,2}(\R^N)$, $(y_{k} ^{(n)})_{k\in\N} \subset \R^N$, $(j_{k} ^{(n)})_{k\in\N}
		\subset \Z$ with $n \in\N$, and disjoint sets
		$\N_0,\N_{+\infty},\N_{-\infty}\subset\N$, such that, for a
		renumbered subsequence of $(v_k)$,
		\begin{eqnarray}
		\label{w_n*}
		&&
		2^{-\frac{N-2}{2}j_k^{(n)}}v_k(2^{-j_k^{(n)}}\cdot+y_k^{(n)})\rightharpoonup 	w^{(n)},\;
		n\in\N,
		\\
		\label{separates*}
		&&
		|j_{k}^{(n)} -  j_{k}
		^{(m)}|+|2^{j_k^{(n)}}(y_{k} ^{(n)} - y_{k} ^{(m)})|\to \infty
		\mbox{ for } n \neq m,
		\\
		\label{norms*}
		&&
		\sum_{n\in \N} \|\nabla w
		^{(n)}\|_{2}^2 \le
		\limsup\|\nabla v_k\|_{2}^2,
		\\
		\label{BBasymptotics*}
		&&
		v_{k}  -  \sum_{n\in\N}
		2^{\frac{N-2}{2}j_k^{(n)}} w^{(n)}(2^{j_k^{(n)}}(\cdot-y_k^{(n)}))
		\to 0\; \mbox{in } L^{2^*}(\R^N),
		\end{eqnarray}
	 and, the series above 	converges in $\dot H^{1,2}(\R^N)$ unconditionally and uniformly
		with respect to $k$.
		\par Moreover,
		$1\in\N_0$, $y_{k} ^{(1)} =0$; $j_k^{(n)}=0$ whenever $n\in\N_0 $;
		$j_k^{(n)}\to -\infty$ (resp. $j_k^{(n)}\to +\infty$) whenever
		$n\in\N_{-\infty}$ (resp. $n\in\N_{+\infty}$); and $y_k^{(n)}=0$
		whenever $|2^{j_k^{(n)}}y_k^{(n)}|$ is bounded.
	\end{theorem}
Note that the unconditional convergence of the series is not stated in the original version of the theorem, but can be easily inferred from the proof. This omission has been remedied in the Banach space version of the theorem in \cite{SoliTi}. This remark applies also to the be based on the profile decomposition in $H^{1,2}(\R^N)$ from \cite{FT}, Corollary~3.3 (it also can be derived from the result of Solimini in \cite{Solimini})

\begin{proposition}
	\label{acc-shifts}
	Let $u_{k}\in H^{1,2}(\R^N)$ be a bounded sequence. There
	exist $w^{(n)}  \in
	H$, $(y_{k} ^{(n)})_{k\in\N} \subset \Z^N$, $y_{k} ^{(1)} =0$, with  $n\in\N$,
	such that, on a renumbered subsequence,
	\begin{eqnarray}
	\label{w_n-shifts} && u_k(\cdot+y_{k} ^{(n)})\rightharpoonup w^{(n)},
	\\
	&&\label{separates-shifts} |y_{k} ^{(n)}-y_{k} ^{(m)}|\to\infty  \mbox{ for
	} n \neq m,
	\\
	&&\label{norms-shifts} \sum_{n \in \N} \|w ^{(n)}\|_{H^{1,2}}^2 \le
	\limsup \|u_k\|_{H^{1,2}}^2,
	\\
	&&\label{BBasymptotics-shifts} u_{k} - \sum_{n\in\N}
	w^{(n)}(\cdot-y_{k} ^{(n)})  \to 0\, \text{in }
	L^p(\R^N), p\in (2,2^*), 
	\end{eqnarray}
and the series in (\ref{BBasymptotics-shifts}) converges
	in $H^{1,2}(\R^N)$ unconditionally and uniformly in $k$.
\end{proposition}

\bibliographystyle{amsplain}

\begin{thebibliography}{10}
\bibitem{AdiTi} Adimurthi, Cyril Tintarev, On compactness in the Trudinger-Moser inequality, Ann. SNS Pisa Cl. Sci. (5) Vol. XIII (2014), 1-18.
\bibitem{BL} H. Brezis, E.Lieb, A relation between pointwise convergence of functions and convergence of functionals. Proc. Amer. Math. Soc. \textbf{88} (1983), 486-490. 
\bibitem{BN} H. Brezis, L. Nirenberg, Positive solutions of nonlinear elliptic equations involving critical Sobolev exponents. Comm. Pure Appl. Math. \textbf{36} (1983),  437-477.
\bibitem{cianchi} A. Cianchi, E. Lutwak, D. Yang, G. Zhang,  Affine Moser-Trudinger and Morrey-Sobolev inequalities. Calc. Var. Partial Differential Equations \textbf{36} (2009), 419-436.
\bibitem{CwiTi} M. Cwikel, Cyril Tintarev, On interpolation of cocompact imbeddings, Revista Matematica Complutense \textbf{26} (2013), 33-55.
\bibitem{GT} D. Gilbarg, N. S. Trudinger, Elliptic partial differential equations of second order. Second edition. Grundlehren der Mathematischen Wissenschaften \textbf{224}. Springer-Verlag, Berlin, 1983. xiii+513 pp. ISBN: 3-540-13025-X.
%
\bibitem{Zhang-p} E.Lutwak, D.Yang, G. Zhang, Sharp affine $L_p$ Sobolev inequalities, J. Diff. Geom.  \textbf{62} (2002), 17--38.
%
\bibitem{Nirenberg59} L. Nirenberg, On elliptic partial differential equations, Ann. SNS Pisa (3) \textbf{13} (1959), 115--162.
%
\bibitem{ST} I. Schindler, K. Tintarev, An abstract version of the concentration compactness
principle, Revista Matematica Complutense, \textbf{15}, 1-20 (2002).
\bibitem{Solimini} S. Solimini, A note on compactness-type properties with respect to Lorentz norms of bounded subsets of a Sobolev space. Ann. Inst. H. Poincar\'e Anal. Non Lin\'eaire 12 (1995), no. 3, 319–337.
\bibitem{SoliTi} S. Solimini, C. Tintarev, Concentration analysis in Banach spaces, Comm. Contemp. Math. \textbf{18} (2016), 1550038 (33
pages).
\bibitem{FT} K. Tintarev, K.-H. Fieseler, Concentration compactness: functional-analytic grounds and applications, Imperial College Press, 2007.
%
\bibitem{4coco} C. Tintarev, Four proofs of cocompactness for Sobolev embeddings, in Functional Analysis, Harmonic Analysis and Image Processing: A collection of papers in honor of Björn Jawerth. Contemporary Mathematics 693, Amer. Math. Soc., Providence, RI, 2017, pp. 321-329.
%
\bibitem{Zhang-1} G. Zhang, The affine Sobolev inequality, 
J. Differential Geom. \textbf{53} (1999), 183-202. 
\end{thebibliography}

\end{document}